\documentclass[11pt,a4paper,leqno]{amsart}
\usepackage[utf8]{inputenc}
\usepackage[T1]{fontenc}

\usepackage{float}
\usepackage{amsfonts,amsmath,graphics,epsfig,latexsym,times}
\usepackage{ae,aecompl}
\usepackage{mathrsfs}
\usepackage[all]{xy}

\newif\ifpix \pixtrue
\usepackage[colorlinks=true]{hyperref}
\hypersetup{ pdfstartview=FitH }

\numberwithin{equation}{subsection}

 \def\PP{{\mathbb{P}}} 
\def\ZZ{{\mathbb{Z}}}    \def\QQ{{\mathbb{Q}}} \def\CC{{\mathbb{C}}}
\def\RR{{\mathbb{R}}}    
\def\cO{{\mathcal{O}}}
\def\cS{{\mathscr{S}}}

\def\Fut{{\mathfrak{F}}}

\def\Kc{{\mathscr{K}}}
\def\Kr{{\mathscr{R}(\Sbar)}}

\def\cL{{\mathscr{L}}}
\def\cX{{\mathcal{X}}}
\def\cXhat{{\widehat {\mathcal{X}}}}
\def\cM{{\mathcal{M}}}

\def\cE{{\mathscr{E}}}
\def\cMhat{{\widehat {\mathcal{M}}}}
\def\cMbar{{\overline {\mathcal{M}}}}

\def\cY{{\mathscr{Y}}}

\def\cF{{\mathcal{F}}}

\def\cP{{\mathcal{P}}}

\renewcommand{\epsilon}{\varepsilon}

\newcommand{\SU}{\mathrm{SU}}

\newcommand{\CP}{\mathbb{CP}}

\newcommand{\del}{\partial}
\newcommand{\delb}{\bar\partial}

\renewcommand{\Im}{\mathrm{Im}}

\newcommand{\chiorb}{\chi^{{orb}}}

\newcommand{\Xbar}{{\overline{X}}}

\newcommand{\Xhat}{{\widehat{X}}}

\newcommand{\Sbar}{\overline{\Sigma}}
\newcommand{\cXbar}{\overline{\cX}}

\newcommand{\parmu}{\mathrm{par}\mu}

\newcommand{\orbideg}{\mathrm{orb\,deg}}
\newcommand{\orbic}{c_1^{orb}}
\newcommand{\pardeg}{\mathrm{par\,deg}}

\newcommand{\hkto}{\hookrightarrow}
\newcommand{\resp}{\emph{resp. }}

\newcommand{\orb}{\mathrm{orb}}

\newtheorem{lemma}[equation]{Lemma}
\newtheorem{prop}[equation]{Proposition}
\newtheorem{cor}[equation]{Corollary}
\newtheorem{theo}[equation]{Theorem}
\newtheorem{theointro}{Theorem}

\theoremstyle{definition}
\newtheorem{dfn}[equation]{Definition}

\theoremstyle{remark}
\newtheorem{rmk}[equation]{Remark}

\newtheorem{example}[equation]{Example}
\newtheorem*{rmk*}{Remark}
\newtheorem*{rmks*}{Remarks}

\date{March 21st 2013}

\title{K-stability and parabolic stability}
\subjclass[2000]{Primary 32Q26; Secondary 53C55, 58E11, 14J26, 14H60}

\author{Yann Rollin}
\address{Yann Rollin, Laboratoire Jean Leray (UMR 6629), Universit\'e de Nantes}
\email{yann.rollin@univ-nantes.fr}

\begin{document}
{\Huge \sc \bf\maketitle}
\begin{abstract}
Parabolic structures with rational weights encode certain iterated blowups
of geometrically ruled surfaces. In this paper, we show that the three
notions of parabolic
polystability, $K$-polystability and existence of constant scalar
curvature Kähler metrics on the iterated blowup are equivalent, for certain polarizations
close to the boundary of the Kähler cone.
\end{abstract}

\section{Introduction}
\label{sec:intro}
The Calabi program is concerned with finding canonical metrics on 
Kähler manifolds. The idea is to look for critical points of the \emph{Calabi
functional}, i.e. the
$L^2$-norm of the scalar curvature, within a prescribed Kähler class. Such
metrics are called \emph{extremal metrics}. The existence problem for
extremal metrics is  open, even for complex surfaces.
The Donaldson-Tian-Yau conjecture roughly says that the existence of extremal metrics with
integral Kähler class should
be equivalent to some
algebro-geometric notion of stability of the corresponding polarized
complex manifold.

The Euler-Lagrange equation for an extremal metric $g$ is equivalent
to the fact that $(\delb s)^\sharp$ --- the $(1,0)$-component of the gradient of the scalar
curvature of $g$ --- is a holomorphic vector field. If the complex
manifold does not carry any
nontrivial holomorphic vector field, a Kähler metric is extremal if
and only if it has constant scalar curvature.
It seems reasonable, at first, to focus on this ``generic'' case, thus limiting
our study to \emph{constant scalar curvature Kähler metrics} (we shall use
the acronym \emph{CSCK} as a shorthand).

Ruled surfaces are an excellent probing playground for the
Donaldson-Tian-Yau conjecture. In this paper, we
study iterated blowups of ruled surfaces encoded by parabolic
structures. Our main result, stated below, shows that  the Donaldson-Tian-Yau
conjecture holds for such class of surfaces and certain
polarizations. In addition, we prove that stability of parabolic
bundles plays a fundamental role in the picture.  
The rest of the introduction will be devoted to explain the relevant
definitions.
\begin{theointro}
\label{theo:A}
Let $\cX\to \Sigma$ be  a parabolic geometrically ruled surface
with rational weights and $\cXhat\to \cX$ the iterated blowup encoded
by the parabolic structure.

If $\cXhat$ has no nontrivial holomorphic
vector fields, the following properties are equivalent:
 \begin{enumerate}
 \item  $\cXhat$ is basically CSCK,
\item $\cXhat$ is basically K-stable, 
\item  $\cX\to\Sigma$ is parabolically stable.
 \end{enumerate}
\end{theointro}

\subsection{Parabolic ruled surfaces}
A geometrically
ruled surface is obtained as the
projectivization $\cX=\PP(E)$ of some 
holomorphic complex vector bundle of rank~$2$  over a closed Riemann surface
 $E\to\Sigma$ and is endowed with a canonical projection $\pi_\Sigma:\cX\to \Sigma$. 
More generally, a ruled surface $\cXhat$ can be described as an iterated
blowup $\pi_{\cX}:\cXhat\to \cX$ of a geometrically ruled
surface $\pi_\Sigma:\cX \to \Sigma$. 

A \emph{parabolic structure}  on a geometrically ruled surface $\pi_\Sigma:\cX\to
\Sigma$ consists of the following data:
\begin{itemize}
\item A  finite set of distinct marked points $y_1,\cdots, y_m\in
  \Sigma$; 
\item marked points   $x_1,\cdots ,x_m\in\cX$ such that
 $\pi_\Sigma(x_j)=y_j$; 
\item real numbers $\alpha_1,\cdots,\alpha_m \in (0,1)$ associated to
  each marked point and called the weights of the marked points. 
\end{itemize}
The geometrically ruled surface together with its parabolic structure is simply
called a \emph{parabolic ruled surface}.

We consider smooth holomorphic curves $S\subset \cX$ such that
$\pi_\Sigma|_S:S\to \Sigma$ has degree $1$, in other words, \emph{holomorphic
sections} of $\cX\to\Sigma$. The \emph{parabolic slope} of $S$ is
defined by the formula
$$\parmu(S) = [S]^2 +\sum_{x_j\not\in S}\alpha_j - \sum_{x_j\in S}\alpha_j,
$$
where $[S]\in H_2(\cX,\ZZ)$ is the homology class of $S$ and $[S]^2$
its self-intersection. In the rest of this text, the homology class of
a curve $S$ will be denoted $S$ as well, without using the brackets.

A parabolic ruled surface is \emph{stable}  if  
  $\parmu(S)
>0$ for
every holomorphic section~$S$.
More generally, we say that 
a parabolic ruled surface  $\cX\to \Sigma$ is \emph{polystable}, if it is
 stable, or if
 there are two non-intersecting holomorphic sections $S_-$ and $S_+$
 with vanishing parabolic slope
  (i.e. sections such that $S_+\cdot
S_-=0$ and $\parmu(S_\pm)=0$).

\begin{rmk}
A parabolic structure on $\cX=\PP(E)\to \Sigma$ gives a line
$x_j\subset E_{y_j}$. This data together with the choice of a pair of weights $0\leq
\beta_1^j<\beta_2^j<1$ such that $\beta^j_2-\beta^j_1=\alpha_j$ for
each point $y_j$  defines a
parabolic structure on the vector bundle $E\to\Sigma$ in the sense of
Mehta-Seshadri~\cite{MehSes80}. With our conventions we have
$\parmu (S)= \pardeg(E) - 2\pardeg(L)$, where $\pardeg$ is the
parabolic degree of a parabolic bundle in the sense of
Mehta-Seshadri and $L$ is the  line sub-bundle corresponding to $S$.

By definition,  the  notions of parabolic  stability for a parabolic
ruled surface $\cX\to \Sigma$ are
equivalent to the various notions of parabolic stability in the sense
of Mehta-Seshadri for the underlying parabolic vector bundle $E\to \Sigma$ (cf. \cite{RolSin05}
for more details).
\end{rmk}

\subsection{Iterated blowups of a parabolic ruled surface with
  rational weights}
\label{secitbup}
Our main result deals with parabolic structures with rational
weights. We shall use the conventions
$\alpha_j=\frac{p_j}{q_j}$ where $0<p_j<q_j$ with $p_j$ and $q_j$
coprime integers.

 In such situation, the marked points and rational weights define an iterated
blowup  $\cXhat\to \cX$ introduced in \cite{RolSin05}. We recall the
construction as it is an essential ingredient of this paper. 
In order to simplify the notations, we pretend that the parabolic
structure on $\cX$ is reduced to a single point $y\in \Sigma$; let $x$ be the corresponding point in $F = \pi^{-1}(y)$ and let
$\alpha = \frac pq$ be the weight. 

 The first step is to
blowup the point $x$, to get a diagram of the form
\begin{equation*}
\xymatrix{
{}\ar@{-}[rr]^{-1}_{\hat F} && *+[o][F-]{}
\ar@{-}[rr]^{-1}_{\hat E} &&{}
}
\end{equation*}
Here the edges represent rational curves, the number above each edge  
is the self-intersection of the curve and
the hollow dots represent transverse intersections with intersection number
$+1$. The curve  $\hat F$
 is the proper transform
 of $F$, whereas the other component $\hat E$ is the exceptional divisor
of
the blowup at $x$.

By blowing-up the intersection point  of $\hat F$ and $\hat E$ we get the
diagram
\begin{equation*}
\xymatrix{
{}\ar@{-}[rr]^{-2} && *+[o][F-]{} 
\ar@{-}[rr]^{-1} && *+[o][F]{}
\ar@{-}[rr]^{-2} &&{} 
}
\end{equation*}
The $-1$-curve above  has exactly two intersection points with
the rest of the string. 
We can decide to blowup either one of them and we carry on with this
iterative procedure, blowing-up at each step one of the two
intersection point of the $-1$-curve. 
 After a finite number  of blowups we obtain an iterated blowup
$$\pi_{\cX}:\cXhat\to \cX$$
with $\pi_{\cX}^{-1}(F)$ given by the  configuration of
curves below:
\begin{equation*}
\xymatrix{
{}\ar@{-}[r]^{-e^-_1}_{E^-_1} & *+[o][F-]{}
\ar@{-}[r]^{ -e_{2}}_{E^-_2} &  *+[o][F-]{}
\ar@{--}[r] &  *+[o][F-]{}
\ar@{-}[r]^{-e^-_{k-1}}_{E^-_{k-1}} &  *+[o][F-]{}
\ar@{-}[r]^{-e^-_k}_{E^-_k} &  *+[o][F-]{}
\ar@{-}[r]^{-1}_{E_0} &  *+[o][F-]{}
\ar@{-}[r]^{-e^+_{l}}_{E^+_l} &  *+[o][F-]{}
\ar@{-}[r]^{-e^+_{l-1}}_{E^+_{l-1}} &  *+[o][F-]{}
\ar@{--}[r] &  *+[o][F-]{}
\ar@{-}[r]^{-e^+_{2}}_{E^+_2} &  *+[o][F-]{}
\ar@{-}[r]^{ -e^+_{1}}_{E^+_1} & 
}
\end{equation*}
where the curves $E^\pm_j$ have 
 self-intersection $-e^\pm_j \leq -2$ and $E^-_1$  is the proper
 transform of the
 fiber $F =\pi^{-1}_{\Sigma}(y)\subset \cX$.
It turns out that there is
exactly one way to perform the iterated blowup so that the integers $e^-_j$ are given by
the continued fraction expansion of $\alpha$:
\begin{equation}\label{e1.844}
\alpha = \frac pq =  \cfrac{1}{e^-_1-\cfrac{1}{e^-_2-\cdots\cfrac{1}{e^-_k}}};
\end{equation}
Then the $e^+_j$'s are given by the continued fraction
\begin{equation}\label{e20.844}
1-\alpha = \frac {q-p}{q} = \cfrac{1}{e^+_1-\cfrac{1}{e^+_2-\cdots
\cfrac{1}{e^+_l}}}.
\end{equation}
Note that these expansions are unique since we are assuming $e^\pm_j \geq 2$.

If the parabolic structure has more marked points, we perform iterated
blowups in the same manner for each marked point and correponding
weight.

\subsection{From parabolic to orbifold ruled surfaces}
Contracting the strings of  $E^\pm_j$-curves in $\cXhat$ gives an
orbifold surface $\cXbar$ and $\pi_{\cXbar}:\cXhat\to\cXbar$ is the 
minimal resolution.
Replacing the marked points $y_j$ of $\Sigma$ with orbifold
singularities of order $q_j$, we obtain an orbifold Riemann surface
$\Sbar$. 
It turns out that there is a holomorphic map of orbifolds 
$$\pi_{\Sbar}:\cXbar\to
\Sbar$$
 which gives $\cXbar$ the structure of a geometrically
ruled orbifold surface. All these facts are detailed in  \cite{RolSin05} where
the structure of the orbifold singularities is studied precisely.

\subsection{Near a boundary ray of the Kähler cone}
\label{sec:bkahl}
 The positive ray 
$$\Kr=\{\gamma \in H^2_{orb}(\Sbar,\RR), \; \gamma\cdot [\Sbar]>0\}
$$
is by definition the entire  Kähler cone of the orbifold Riemann
surface $\Sbar$.  
For practical reasons, the image of $\Kr$ under the canonical injective maps 
$$H^2_{orb}(\Sbar,\RR)\stackrel {\pi_{\Sbar}^*}\hkto
H^2_{orb}(\cXbar,\RR) \stackrel{\pi_{\cXbar}^*}\hkto H^2(\cXhat,\RR)$$
shall be denoted by $\Kr$ as well. So, depending on the context,
$\Kr$ will represent a ray in $H^2_{orb}(\Sbar,\RR)$,
$H^2_{orb}(\Xbar,\RR)$ or $H^2(\cXhat,\RR)$.
\begin{rmk}
  The notation $H^k_{orb}$ stands for the orbifold
  De Rham cohomology. Here, we emphasize the fact that we like to represent
cohomology  classes by closed differential forms which are smooth in the
  orbifold sense. However, the notation is unimportant, as there is a canonical
  isomorphism $H^k_{orb}(\cXbar,\RR)\simeq H^k(\cXbar,\RR)$ with the
  standard singular cohomology. The proof of this property boils down to the
  fact that there is a local Poincaré lemma in the context of orbifold
  De Rham cohomology.
\end{rmk}

Let $\Kc(\cXbar)$ and $\Kc(\cXhat)$ be the Kähler cones of the orbifold
$\cXbar$ and of $\cXhat$. It is well known that $\cXhat$ and $\cXbar$
are of Kähler type. These cones are therefore nonempty. The following
lemma is more precise, and concerns the Kähler classes
that will be relevant for our results:
\begin{lemma}
\label{lemma:ray} The ray $\Kr$ is
  contained in the closure of $\Kc(\cXhat)\cap H^2(\cXhat,\QQ)$ in $H^2(\cXhat,\RR)$.
In other words, for every open cone $U\subset
H^2(\cXhat,\RR)$ such that $\Kr\subset U$, the cone
$\Kc(\cXhat)\cap H^2(\cXhat,\QQ)\cap U$ is nonempty.
\end{lemma}

The fiberwise hyperplane section of $\cXbar \to \Sbar$ defines a holomorphic orbifold line
bundle  denoted $\cO_{\cXbar}(1)\to \Xbar$ (the construction is
completely similar to the case of smooth geometrically ruled surfaces).

Like in the smooth case, one can construct a Hermitian metric
 $h$ on
$\cO_{\cXbar}(1)\to \cXbar$ with curvature $F_h$, such the closed
$(1,1)$ form $\omega_h= \frac i{2\pi} F_h$ restricted to any fiber of
$\cXbar\to \Sbar$ is a Kähler form. We may even assume that the
restriction of $\omega_h$ to the fibers agrees with the Fubini-Study
metric on $\CP^1$. Notice that with our conventions
$$
[\omega_h]=\orbic (\cO_{\cXbar}(1)),
$$
where $\orbic \in H^2_{orb}(\cXbar,\RR)$ denotes the first (orbifold) Chern class of an orbifold
complex line bundle. Modulo the isomorphism $
H^2_{orb}(\cXbar,\RR)\simeq H^2(\cXbar,\RR)$ between DeRham orbifold
cohomology and singular cohomology, one can show that orbifold Chern
classes are  rational. 

Let $\Omega_{\Sbar}\in \Kr$ be a Kähler class on $\Sbar$
represented by a Kähler metric with Kähler form $\omega_{\Sbar}$.
We shall assume that $\Omega_{\Sbar}$ is integral, which is always
possible, up to multiplication by a positive constant.
It is easy to check that  for every constant $c>0$ sufficiently small,
the closed $(1,1)$-form
\begin{equation}
  \label{eq:omc}
\omega^{orb}_c=  \pi_{\Sbar}^*\;\omega_{\Sbar} +c\;\omega_h   
\end{equation}
is definite positive on $\cXbar$. Thus $\omega^{orb}_c$ defines a  Kähler orbifold metric on
$\cXbar$ with
Kähler class 
\begin{equation}
  \label{eq:orbkclass}
\Omega^{orb}_c=\bar\pi^*\Omega_{\Sbar} +
c\cdot \orbic (\cO_{\cXbar}(1)).  
\end{equation}

Assuming again that the parabolic structure has exactly one marked
point,  we consider $(1,1)$-cohomology class on $\cXhat$ given
by
\begin{equation}
  \label{eq:smoothkclass}
\Omega = \pi^*_{\cXbar}\Omega_c^{orb}+ \sum_{j=1}^k c_j^-[E_j^-] + \sum_{j=1}^l c_j^+[E_j^+],  
\end{equation}
where  $c^\pm_j \in \RR$ and $[E_j^\pm]\in H^2(\cXhat,\ZZ)$ denotes
the Poincaré dual of $E_j^\pm\in H_2(\cXhat,\ZZ)$.
Here, the constants $c^\pm_j$ are uniquely determined by the values of
$\Omega\cdot E^\pm_j$, since the intersection matrix of the
$E^\pm_j$-curves is invertible. 
Then we have the following result:
\begin{lemma}
\label{lemma:kahcldeg}
Given $c>0$, there exists 
$\epsilon >0$,  such that every cohomology class $\Omega$ given by
\eqref{eq:smoothkclass} 
and satisfying $0<\Omega\cdot E^\pm_j<\epsilon$  is a Kähler class.
\end{lemma}
\begin{proof}
Kodaira that showed
that (smooth) Kähler manifolds are stable under blowup.
Kodaira's argument can be adapted to the orbifold setting,
 and the proof is  nearly identical.
  Following \cite{ArePac06} and using the scalar-flat ALE metrics of
Calderbank-Singer \cite{CalSin04}, one can construct a Kähler metric
$\omega$ on $\cXhat$ by gluing $\pi^*_{\Xbar}\omega_c^\orb$ and a
small copy of one to the Calderbank-Singer metrics.
Every Kähler class such that the areas $\Omega\cdot E^\pm_j$ are
sufficiently small is obtained in this way.
\end{proof}

\begin{proof}[Proof of Lemma \ref{lemma:ray}]
An element of $\Kr$ is represented by a Kähler class
$\Omega_{\Sbar}$. 
The constant $c$ and $\Omega\cdot [E^\pm_j]$ that appear in the above
discussion can be chosen to be rational and we may assume that we have
a Kähler class 
$\Omega\in H^2(\cXhat,\QQ)$. 
It is now obvious that $\Omega$ is arbitrarily close to the pullback
of $\Omega_{\Sbar}$  for $c$ and $\Omega\cdot [E^\pm_j]$ sufficiently
small. The result follows for the case where the parabolic structure
has exactly one point. The general case is an obvious generalization.
\end{proof}

\begin{dfn}
\label{dfn:basic}
If  there exists an open cone $U$ in $H^2(\cXhat,\RR)$, containing the
ray $\Kr$, with the property that any Kähler class in
$U\cap \Kc(\cXhat)$ can be represented by a CSCK  (\resp extremal) metric,
we say that the iterated blowup $\cXhat$ is
\emph{basically CSCK} (\resp \emph{extremal}).

If  there exists an open cone $U$ in $H^2(\cXhat,\RR)$, containing the
ray $\Kr$, with the property that any rational Kähler class in
$U\cap \Kc(\cXhat)$ is K-stable,
we say that the iterated blowup $\Xhat$ is \emph{basically K-stable}.
\end{dfn}
More generally, any property $\mathscr{P}(\Omega)$ depending on the
choice of a cohomology 
class $\Omega\in H^2(\cXhat,\RR)$ is said to be \emph{basically
  satisfied}, if it
holds for every $\Omega$ contained in  a sufficiently small cone about
the ray $\Kr$. In other words, if $\mathscr{P}$ holds for every
$\Omega$ sufficiently 
close to  a \emph{basic class}. This explains my choice of
terminology; peharps there are better choices and I am open to
suggestions. 

\begin{example}
There is no general existence theory for extremal metrics. An
exciting approach  for fibrations in
various contexts was adopted by  Hong, Fine and
Brönnle \cite{Hon99,Fin04,Bro13}. Their idea is to construct
approximate extremal metrics by making the base of the fibration huge, which is
sometimes refered to as taking \emph{an adiabatic limit}. Then the
extremal metric 
is obtained by perturbation theory. The results aforementioned show
that the fibration  under consideration is \emph{basically extremal}
in the sense of Definition \ref{dfn:basic}. 
\end{example}

\begin{rmk}
The condition of $K$-stability may be defined for varieties polarized
by a rational Kähler class.
From a
  more down to earth point of view, a rational class becomes an
  integral Kähler class $\Omega$
  after multiplication by a suitable positive integer.
 The class $\Omega$ defines an ample holomorphic line
  bundle $L_\Omega\to \cXhat$ with $c_1(L_\Omega)=\Omega$ and the
  condition of $K$-stability for the original rational polarization is
  equivalent to the usual notion of $K$-stability for
  $(\cXhat,L_{\Omega})$ (cf. \S\ref{sec:tc} for more details).
\end{rmk}

\subsection{Comments and proof of  Theorem \ref{theo:A}}
Our main result is an attempt to solve the conjecture made 
in \cite{RolSin05}. Loosely
speaking, we expect a correspondence 
between the two classes of objects represented in the following
diagram:

\vspace{10pt}

\begin{center}
\begin{tabular}{lcr}
 \fbox{\parbox{1.8in}{Parabolically stable ruled surfaces $\cX\to \Sigma$}} 
& $\sim$ &\fbox{\parbox{2in}{CSCK metrics on the corresponding
     iterated blowup   $\cXhat$
}}
\end{tabular}
\end{center}
Theorem \ref{theo:A} shows that the answer to the conjecture is positive,
 under some mild assumptions, provided we consider only certain Kähler
classes on $\cXhat$ close to the boundary ray $\Kr$ of the Kähler cone. 

We should point out that when the parabolic structure is empty,
i.e. when $\cXhat=\cX = \PP(E)$ is a geometrically ruled surface, the problem is completely
understood~\cite{ApoTon06}. In this case, the result of Apostolov and
T{\o}nnesen-Friedman says that
$E\to\Sigma$ is a polystable holomorphic bundle if and only if
$\PP(E)$ is CSCK, for \emph{any} Kähler class.

Notice that the conjecture deals with highly non generic ruled
surfaces. Indeed, the complex structures of  iterated blowups $\cXhat\to\cX$ encoded by parabolic
structures are very special.
It is tempting to believe that our result could be used as
the very first step 
toward a proof of the general Donaldson-Tian-Yau conjecture for ruled
surfaces. Here,  some kind of deformation theory and continuity method
is needed. Important progress shall
 be made for completing this program, especially for dealing with  the difficult compactness issue of the
relevant moduli spaces.

\begin{proof}[Proof of Theorem \ref{theo:A}]
 $(1)\Rightarrow (2)$ is an 
immediate consequence of Stoppa's result~\cite{Sto09}. 

 $(3)\Rightarrow (1)$   is essentially contained in the joint
work of the author with Michael Singer~\cite{RolSin05,RolSin09,RolSin09-2} plus some slight improvements
explained at~\S\ref{sec:glue}.  More precisely the result follows from
point $(1)$  in Theorem~\ref{theo:glue}.

 $(2)\Rightarrow (3)$ was the missing piece of the
puzzle that completes the full picture. We shall prove that if $\cX\to\Sigma$ is not parabolically
stable, one can construct  destabilizing test configuration
as proved in Corollary~\ref{cor:key}.
This requires a delicate computation for the Futaki invariant at~\S\ref{sec:sfut}.
\end{proof}

\subsection{Acknowledgements}
The author would like to thank G\'abor
Sz\'ekelyhidi for some stimulating discussions during the fall 2012, at  the MACK5
conference in Rome.

\section{Extremal ruled surfaces and gluing theory}
\label{sec:glue}
\subsection{Application of the Mehta-Seshadri theorem}
Let $\cX\to\Sigma$ be  a parabolic geometrically ruled surface with
rational weights. If $\cX\to\Sigma$ is parabolically polystable, it is
a flat $\CP^1$ bundle on the complement of the fibers
$\pi_\Sigma^{-1}(y_j)$ by Mehta-Seshadri
theorem \cite{MehSes80} and the monodromy of the flat
connection is given by a morphism
$\rho:\pi_1(\Sigma \setminus \{y_j\})\to\SU_2/\ZZ_2$. In
addition, if $l_j$ is the 
homotopy class of a loop in $\Sigma \setminus \{y_j\}$ winding once
around $y_j$, then $\rho(l_j)$ is given by the matrix
\begin{equation*}
\left (  \begin{array}{cc}
    e^{i\pi\alpha_j} & 0 \\
0&     e^{-i\pi\alpha_j}
  \end{array}
\right )
\end{equation*}
up to conjugation. In particular $\rho(l_j)$ has order $q_j$ and the
morphism descends to 
$$\rho:\pi_1^{orb}(\Sbar)\to \SU_2/\ZZ_2,$$
 as the orbifold fundamental group $\pi_1^{orb}(\Sbar)$ is just deduced from $\pi_1(\Sigma\setminus\{y_j\})$ by adding the relation~$l_j^{q_j}=1$.

Now, the orbifold Riemann surface $\Sbar$ admits an orbifold metric
$g_{\Sbar}$ of constant curvature 
in its conformal class, unless it is a ``bad'' orbifold in the sense
of Thurston. That is if $\Sbar$ is a \emph{teardrop} or a
\emph{football}\footnote{Using  the more politically correct
  term \emph{northern-American-football} may be a safer
  option. 
} 
with two
singularities of distinct 
orders. 

\begin{rmk}
\label{rmk:good}
In fact $\Sbar$ cannot be bad if $\cX\to\Sigma$ is
polystable. Indeed, assume that $\Sigma\simeq \CP^1$ and that the
parabolic structure has exactly one marked point with weight
$\alpha=p/q$. Let $l$ be the homotopy class of a loop winding once
around the parabolic point of $\CP^1$. By the Mehta-Seshadri theorem,
$\rho(l)$ has order $q$. But $l$ is trivial, since $\CP^1$ with one
puncture is contractible. It follows that $q=1$ which is impossible.
A similar argument shows that $\Sbar$ cannot be a football with two
singularities of distinct orders.  
\end{rmk}

Since the monodromy acts isometrically on 
$\CP^1$ endowed with the Fubini-Study metric,
the twisted product $\Sbar\times_\rho \CP^1$  carries a
local product deduced $g_{\Sbar}$ and
$g_{FS}$. Adjusting the metrics on each factor by a constant. In
conclusion of the above discussion and Remark \ref{rmk:good}, we get
the following lemma:
\begin{lemma}
\label{lemma:csckms}
  If $\cX\to\Sigma$ is parabolically polystable, $\Sbar$ and $\cXbar$
admit orbifold CSCK metrics in every Kähler classes.
\end{lemma}
This construction 
was the key argument used in
\cite{RolSin05,RolSin09,RolSin09-2} together with the  Arezzo-Pacard
gluing theory \cite{ArePac06}, for  producing
CSCK metrics on the desingularization $\cXhat$ of $\cXbar$. The point
is that the local resolution of isolated singularity that occur 
in $\cXbar$ admit  scalar-flat Kähler metrics deduced from the
 Calderbank-Singer ALE scalar-flat Kähler metrics 
\cite{CalSin04}. 

The gluing theorem that we shall use can be stated as follows.
\begin{theo}
\label{theo:apunif}
  Let $\cXbar$ be a CSCK orbifold surface with Kähler class $\Omega^{orb}$
  and isolated singularities $z_i$ modelled on $\CC^2/\Gamma_i$, where
  $\Gamma_i$ is a finite cyclic subgroup of $U(2)$. Let $\pi_{\cXbar}:\cXhat\to
  \cXbar$ be the minimal resolution.
Then  $\cXhat$  admits extremal metrics in every Kähler class sufficiently
close to~$\pi_{\cXbar} ^*\Omega^{orb}$. 
\end{theo}
\begin{proof}
In the case where $\cXhat$ has no nontrivial holomorphic vector field,
  the result  is essentially an application of Arezzo-Pacard gluing
  theorem \cite{ArePac06} to $\Xbar$ and the Calderbank-Singer
  metrics \cite{CalSin04}. 

Arezzo-Pacard gluing theorem actually provides only  a
one parameter family of CSCK metrics, by gluing  in a copy of
 one particular Calderbank-Singer metric with scale $\epsilon$. One can improve
the Arezzo-Pacard gluing theory, working uniformly with all (a finite
dimensional smoothly varying family)
Calderbank-Singer metrics and the result follows.

In the case where $\cXhat$ admits non trivial holomorphic vector
fields, one can prove the same result working with the equation of
extremal metrics instead. It suffices to work modulo a maximal compact
torus of the isometry group of $\cXbar$. This approach has been
successfully implemented by Tipler \cite{Tip11}. Again one has to be
extra careful to get a uniform result, not only a one parameter family.
\end{proof}

We deduce the following theorem:
\begin{theo}
\label{theo:glue}
Suppose that $\cX\to \Sigma$ is a parabolic ruled
surface and  $\Omega^{orb}$ be an orbifold Kähler class on $\cXbar$. 

If $\cX\to\Sigma$ is parabolically polystable, then
 every Kähler class of $\cXhat$ sufficiently close to
$\pi_{\cXbar}^*\Omega^{orb}$ contains an extremal metric.

If $\cX\to \Sigma$ is parabolically stable then 
$\cXhat$ has no nontrivial holomorphic vector fields
and the extremal metric must be CSCK.
\end{theo}
\begin{proof}
If $\cX\to\Sigma$ is polystable, $\Xbar$ admits a CSCK metric with
Kähler class $\Omega^{orb}$ by Lemma~\ref{lemma:csckms}.
The existence of extremal metrics for Kähler classes $\Omega$
sufficiently close to $\Omega^{orb}$ on $\cXhat$ follows from
 Theorem \ref{theo:apunif}.

If $\cX\to \Sigma$ is parabolically stable, $\pi^{orb}_1(\Sbar)$ acts with no fixed points on $\CP^1$ via the
morphism $\rho$ (this is a part of the Mehta-Seshadri theorem \cite{MehSes80}). 
This implies that $\Sigma$ is not $\CP^1$ with two marked points,
otherwise, $\cX\to\CP^1$ would be at best polystable.
In particular,  $\Sbar$ has no nontrivial
holomorphic vector field.
Following  \cite{RolSin09}, we deduce that $\cXhat$ has no nontrivial
holomorphic vector field either. In conclusion, every extremal metric on $\cXhat$
must be CSCK.
\end{proof}

  In \cite{RolSin09-2}, it was proved that (under some mild
 technical  assumptions), one can always find Kähler classes close to
 $\pi^*_{\cXbar}\Omega^{orb}$ which are represented by CSCK metrics, even when
 $\cX\to\Sigma$ is polystable but not stable. The technique is based on
 a refinment of Arezzo-Pacard gluing theory in presence of
 obstructions~\cite{ArePac09}. 

A computation of the
 Futaki invariant (cf. \S\ref{sec:fut}) allows to deduce this result
 from Theorem~\ref{theo:glue}
 in a simpler way. Indeed, an extremal metric is CSCK if and only if
 its Futaki invariant vanishes. The following theorem 
 also shows that there are always  
 Kähler classes near  $\pi^*_{\cXbar}\Omega^{orb}_c$ which are represented by non CSCK extremal metrics,
 when the parabolic structure is non trivial:
 \begin{theo}
   Suppose that $\cX\to \Sigma$ is a parabolically polystable ruled
surface which is not parabolically stable. We are also assuming that
the parabolic structure is not trivial and that
$\Sbar$ is not a football. 

Then the Lie algebra of holomorphic vector field of $\cXhat$ has
dimension $1$ and is spanned by some vector field $\Xi$. In addition, there exists an open cone $U\subset H^2(\cXhat,\RR)$ such that
$\Kr\subset U$ with the property that the equation $\Fut(\Xi,\cdot)=0$ cuts $U\cap
\Kc(\Xhat )$ along a non empty regular hypersurface containing $\Kr$
in its closure.
 \end{theo}
 \begin{proof}
   The condition of stability implies that $\Sbar$ cannot be a
   teardrop or a football with two singularities of distinct orders
   (cf. Remark \ref{rmk:good}). If $\Sbar$ is a footbal with two
   singularities of the same order, $\Xbar$ and $\Xhat$ are actually
   toric and the Lie algebra of holomorphic vector fields is two
   dimensional. In all the other cases, $\Sbar$ has no nontrivial
   holomorphic vector fields and it follows that the
   Lie algebra is one dimensional (cf. \cite{RolSin09-2}).
The property of the Futaki invariant then follows from Lemma
\ref{lemma:futs2}. 
 \end{proof}

\section{Unstable parabolic ruled surfaces and test configurations}
The aim of this section of to prove the statement $(2)\Rightarrow (3)$
of Theorem \ref{theo:A}.
Let $\cX\to \Sigma$ be a geometrically ruled surface with parabolic
structure and rational weights as in Theorem \ref{theo:A}. In this
section, we shall assume that $\cXhat$ carries no nontrivial
holomorphic vector fields and that $\cX\to\Sigma$ is parabolically unstable.
Then, there is a holomorphic section of $\cX\to\Sigma$, denoted $S$, such that 
 $\parmu(S)\leq 0$.

\subsection{Holomorphic sections and Extensions}
By definition of a geometrically ruled
surface, $\cX=\PP(E)$, where $E\to \Sigma$ is a rank $2$ holomorphic
vector bundle. The section $S$ corresponds to a 
holomorphic line bundle $L_+\subset E$ and we have an exact sequence
of holomorphic vector bundles 

$$
\xymatrix{
0\ar[r] &  L_+\ar[r]\ar[dr]& \ar[d]\ar[r]E& \ar[dl]\ar[r] L_-& 0 \\
&&\Sigma &&
}
$$
where $L_-=E/L_+$. The vector bundle $E$ must be an extension
bundle; more precisely, $E$ is defined by a element 
$\tau\in H^1(\Sigma,L_-^*\otimes L_+)$. Such an extension will be denoted
$E=E_\tau$. 

Let 
$U_j$ be an open cover of $\Sigma$ and a cocycle $\tau_{ij}:U_i\cap U_j\to
L_-^*\otimes L_+$ defining $\tau$. 
Let $\cL_\pm\to \CC\times \Sigma$ be 
the holomorphic line bundles obtained as the pullback of $L_\pm$ via
the canonical projection $\CC\times \Sigma\to\Sigma$.
We introduce the extension bundle $\cE\to \CC\times \Sigma$ defined as follows:
 the restriction of $\cE$  to any open set $\CC\times U_j$ is
 isomorphic to
 $(\cL_+\oplus \cL_-)|_{\CC\times U_j}$ and the
transition maps on $\CC\times (U_i\cap U_j)$ are given by 
$$(\lambda,z,l_+,l_-)\mapsto
(\lambda, z, l_++\lambda \tau_{ij}(z)\cdot l_-,l_-),$$
where $\lambda\in\CC$, $z\in U_i\cap U_j$ and $l_\pm$ belong to the
fiber of $L_\pm \to \Sigma$ over $z$.

The restriction of $\cE$ over $\{1\}\times\Sigma$ is
canonically  identified to $E_\tau\simeq E \to\Sigma$ and
the bundle $\cE$ sits in an exact sequence
$$
\xymatrix{
0\ar[r] &  \cL_+\ar[r]\ar[dr]& \ar[d]\ar[r]\cE& \ar[dl]\ar[r] \cL_-& 0 \\
&&\CC\times\Sigma &&
}
$$
There is an obvious $\CC^*$-action defined on 
open sets, and given by $u \cdot (\lambda,z,l_+,l_-)=(u\lambda , z,
u l_+,l_-)$. This action lifts as an linear
$\CC^*$-action on $\cE\to \CC\times \Sigma$.
Its restriction to $\cE_{\{0\}\times \Sigma}$ identified to $L_+\oplus
L_-\to \Sigma$ is the  induced action on the fibers given by 
$u\cdot (l_+,l_-)=(u l_+,l_-)$.

Passing to the projectivization $\cM=\PP(\cE)$, we obtain a ruled manifold
$\cM\to \CC\times \Sigma$. In particular, the line bundle
$\cL_+\subset \cE$
defines a divisor $\cS\subset \cM$ with the property that $\cS\cap
\cM_1$ is identified to $S\subset \cX$ whereas $\cS\cap\cM_0$ is
identified to $S_+\subset \PP(L_+\oplus L_-)$. We summarize our observations in the following lemma:
\begin{lemma}
\label{lemma:cstardefo}
Given a geometrically ruled surface $\cX\to\Sigma$ and a holomorphic
section $S$, there exists a complex manifold $\cM$ endowed with a $\CC^*$-action
  and  a $\CC^*$-equivariant submersive holomorphic map
  $\pi_\CC:\cM\to\CC$, with respect to the standard $\CC^*$-action on $\CC$,
  such that:
  \begin{itemize}
  \item $\cM_1=\pi_\CC^{-1}(1)$ is isomorphic to $\cX\simeq \PP(E)$;
\item $\cM_0=\pi_\CC^{-1}(0)$ is isomorphic to $\PP(L_+\oplus L_-)$
  with the above notations.
\item In $\cM_0$, the corresponding divisors  $S_+$ and $S_-$ are
  respectively the
  attractive and repulsive sets  of fixed points in $\cM$ under the
  $\CC^*$-action.
\item There exists a $\CC^*$-invariant section $\cS$ of $\cM\to
  \CC\times\Sigma$ such that $\cS\cap \cM_0=S_+$ and $\cS\cap \cM_1=S$.
  \end{itemize}
\end{lemma}

\begin{rmk}
\label{rmk:diffeo}
As a consequence of the above lemma, the pair $(\cM_1,S)$
  is diffeomorphic to $(\cM_0,S_+)$. In particular $S^2=S_+^2$.
\end{rmk}

\subsection{Test configurations and the Donaldson-Futaki invariant}
\label{sec:tc}
The definition of $K$-stability involves general \emph{test
configurations} (cf. \cite{Don02}). 
 However \emph{regular} test configuration will be
sufficient for our purpose, that is:
\begin{itemize}
\item a complex manifold $\cM$ with a holomorphic $\QQ$-line bundle $\cP\to\cM$, 
\item a $\CC^*$-action on $\cM$ that lifts to a linear $\CC^*$-action
  on $\cP\to \cM$.
\item $\CC^*$-equivariant submersive holomorphic map
  $\pi_\CC:\cM\to\CC$
\end{itemize}
such that $\cP\to\cM$ is a fiberwise polarization. In other words,  $\cP$ restricted to  $\cM_\lambda=\pi_\CC^{-1}(\lambda)$
  is an ample $\QQ$-line bundle for every $\lambda\in\CC$ (i.e. a
  \emph{polarization}). 

The Donaldson-Futaki invariant is defined in the following way:
the $\CC^*$-action of a test configuration, $\cP\to\cM\to\CC$ induces
a $\CC^*$-action on the central fiber.
  The vector space
of holomorphic sections $V_k=H^0(\cM_0,\cP_0^k)$ is also acted on by
$\CC^*$. The quantity $F(k)= \frac{w_k}{kd_k}$, where $w_k$ is
the weight of the action on $V_k$ and $d_k=\dim V_k$ admits an
expansion $F(k)=F_0 + k^{-1}F_1 + \cO(k^{-2})$ and $F_1$ is the
Donaldson-Futaki invariant of the test configuration. 

A complex manifold polarized by a $\QQ$-line bundle is said to be $K$-stable if for every
test configuration $\cP\to\cM\to\CC$ where it appears as the generic
fiber, we have $F_1\geq 0$ and $F_1=0$ if the test configuration is a
product.

On the other hand, the usual Futaki invariant~\cite{Fut83} is an
object defined in a purely analytical way, on a smooth Kähler manifold
$\cM_0$ with Kähler class $\Omega$.
Given a holomorphic vector field $\Xi$ of type $(1,0)$  and a
Kähler metric $\omega$ on $\cM_0$, the Futaki invariant is given by
$$
\Fut(\Xi,\Omega)= -2 \int _{\cM_0} \Xi\cdot Gs \;d\mu
$$
where $d\mu$ is the volume form, $s$ is the scalar curvature and $G$
is the Green function associated to the metric with Kähler form
$\omega$. It turns out that the Futaki invariant depends only on the
Kähler class $\Omega$, not on its representative $\omega$ used in the
definition. If $\Xi$ vanishes at some point, there exists a smooth
function  $t:\cM_0\to \CC$ such that $\Xi=\del^\sharp t := (\delb 
t)^\sharp$ (cf. \cite{LebSim94}). In other words, $\Xi$  is the $(1,0)$-component of the
gradient of some smooth function. Then one can show the identity
\begin{equation}
  \label{eq:fut}
\Fut(\Xi,\Omega)= \int _{\cM_0} t(\bar s- s)\;d\mu
\end{equation}
where the constant $\bar s$ is the average of $s$.
\begin{rmk}
We are using the opposite sign convention to the one used by
  Donaldson for the Futaki invariant~\eqref{eq:fut}, which explains
  the sign discrepancy  when quoting his result at Proposition~\ref{prop:don}.
\end{rmk}
It was pointed out by Donaldson that the Donaldson-Futaki agrees with
the Futaki invariants up to a
constant  in the
regular case:
\begin{prop} [{\cite[Proposition 2.2.2]{Don02}}]
\label{prop:don}
For a regular test configuration $\cP\to\cM\to\CC$ we have
$$  F_1 = \frac 1{4vol(\cM_0)}\Fut(\Xi,\Omega)$$
where $\Omega$ is the Kähler class defined by the polarization
$\cP_0\to\cM_0$, ${vol(\cM_0)}$ the corresponding volume and $\Xi$ is the Euler vector field of the
$\CC^*$-action on $\cM_0$.
\end{prop}

In Lemma \ref{lemma:cstardefo}, we already produced a regular
$\CC^*$-equivariant family of deformation $\cM\to\CC$ of the ruled
surface  $\cM_1\simeq \PP(\cE)=\cM\to \CC$ with the property that
$\cM_0\simeq \PP(L_+\oplus L_-)$. Relying on this result, we can
easily construct test configurations for the iterated blowup $\cXhat$.

The central fibers $\cM_0$  has two divisors $S_\pm$ determined by
the line bundles $L_\pm$. The manifold $\cM$ also has a divisor
$\cS_+$ defined by the subbundle $\cL_+\subset \cE$ with the property
thatt $S_+=\cS_+\cap \cM_0$. 

The parabolic structure on $\cM_1\simeq \PP(E)=\cX$ consists a finite set of
marked points $x_j$ in distinct fibers of $\cX\to\Sigma$.
Let $X_j$ be the closure in $\cM$ of the orbit of the
points $x_j$ under the $\CC^*$-action. The points $x_j^\lambda=X_j\cap
\cM_\lambda$ and weights 
$\alpha_j$ define a 
parabolic structure on each fiber of $\cM_\lambda$. 
Notice that all the points of the parabolic structure induced on
$\cM_0\to \Sigma$ must belong to $S_\pm$.
\begin{rmk}
\label{rmk:key}
The above construction gives in particular a parabolic structure on
$\cM_0\to\Sigma$. Using Remark~\ref{rmk:diffeo}, we have
$\parmu(S_+)=\parmu(S)$, by definition.
\end{rmk}
Following the
algorithm described at \S\ref{secitbup}, one can make a iterated blowup of
every deformation $\cM_\lambda$ simultaneously. This boils down to perform an iterated
blowup of the curves $X_j$ in $\cM$. Thus, we obtain a blowup
$\cMhat\to \cM$ with the property that  $\cMhat_1 \simeq \cXhat$.
The $\CC^*$-action lifts to a $\CC^*$-action on $\cMhat$ and we
actually have a $\CC^*$-equivariant family of deformations $\cMhat\to
\CC$.

 For a ruled surface
$h^{2,0}=h^{0,2}=0$, hence any class in $H^2$ is of type
$(1,1)$.
The fibration $\cMhat\to \CC$ is smoothly trivial. So the cohomology
spaces $H^2(\cMhat_\lambda,\RR)$ are all identified canonically to
$H^2(\cXhat,\RR)$. We consider the cohomology class $\Omega^{orb}_c$
as a class in $H^2(\cXhat,\RR)$. We saw that for $c>0$ sufficiently
small, the class $\Omega^{orb}_c$ may be perturbed to give  a Kähler
class on $\cXhat$ (cf. \S\ref{sec:bkahl}). The same result applies to
$\Omega^{orb}_c$ understood as a cohomology class on $\cMhat_0$ and we
have the following result:
\begin{lemma}
  For $c>0$ and $\epsilon>0$  sufficiently small, the cohomology classes
  $\Omega$ in Lemma \ref{lemma:kahcldeg} define Kähler classes on
  $\cMhat_\lambda$ for every $\lambda\in\CC$.
\end{lemma}
In particular, if the constants $c$ and $c^\pm_j$ are all chosen
rational, the cohomology class $\Omega$ is rational and it defines a
$\QQ$-line bundle $\cP\to\cMhat$ with the property that $c_1(\cP)=\Omega$.
We summarize our construction in the following proposition:
\begin{prop}
\label{prop:tc}
  Let $\cX\to\Sigma$ be a parabolic ruled surface with rational
  weights and $S$ a holomorphic section. There exists a sufficiently
  small open cone $U\subset H^2(\cXhat,\RR)$ that contains
  $\Kr$ with the property that for every rational Kähler class
  $\Omega\in\Kc(\cXhat)\cap U$, we can define a test configuration
  $\cMhat\to\CC$, polarized by a $\QQ$-line bundle $\cP\to\cM$ with
  the following property:
  \begin{enumerate}
\item $\cMhat\to\CC\times\Sigma$ is an iterated blowup encoded by the
  parabolic structure of the ruled
  manifold $\cM\to\CC\times \Sigma$ given by
  Lemma~\ref{lemma:cstardefo}, endowed with the induced $\CC^*$-action.
  \item The retriction $\cP|_{\cMhat_1}\to \cMhat_1$ is identified to
    $\cXhat$ endowed with a $\QQ$-line bundle of first Chern class $\Omega$,
\item  $\cMhat_0$ is an iterated blowup of
  $\cM_0\simeq \PP(L_+\oplus L_-)$ encoded by the induced
  parabolic. All the parabolic points of $\cM_0$ where the blowups
  occur are located on the 
  sections $S_\pm$ corresponding to $L_\pm$.
  structure.
  \end{enumerate}
\end{prop}

\section{On the Futaki invariant of blownup ruled surfaces}
\label{sec:fut}
In this section, we shall prove the following proposition.
\begin{prop}
\label{prop:futsign}
Let $\cX\to\Sigma$ be a parabolic ruled surface with rational weights. Let
$S$ be a holomorphic section such that $\parmu(S)\leq 0$. Then
for every open cone $U\subset H^2(\cXhat,\RR)$ such that Proposition~\ref{prop:tc} holds, there exists a rational Kähler class $\Omega\in
\Kc(\cXhat)\cap U$ such that corresponding test configuration
$\cP\to\cMhat\to\CC$ has non-positive Donaldson-Futaki invariant.
\end{prop}
Since $\cXhat$ has no nontrivial holomorphic vector field and
$\cMhat_0$ does, the test configuration must be non-trivial. Thus we
get the following corollary  which proves  the statement
$(2)\Rightarrow (3)$ of 
Theorem \ref{theo:A}:
\begin{cor}
\label{cor:key}
  If $\cX\to\Sigma$ is not parabolically stable and $\cXhat$ has no
  nontrivial holomorphic vector fields, then $\cXhat$ is not basically K-stable.
\end{cor}

\begin{proof}[Proof of Proposition \ref{prop:futsign}]
Donaldson proved in~\cite{Don02} that for a regular test 
configuration, the Donaldson-Futaki invariant is actually given by the usual Futaki
invariant of the central fiber. 
This is the case for the  test configuration $\cP\to\cMhat\to\CC$ 
 Computing the Futaki invariant of its central fiber is the goal of
the rest of this section. In particular, the proposition follows from
the Lemmas \ref{lemma:futs1}, \ref{lemma:futs3}, \ref{lemma:futs2} and Remark~\ref{rmk:key}. 
\end{proof}

\subsection{Geometrically ruled surfaces with circle symmetry}

From Proposition~\ref{prop:tc},  we have 
 $\cM_0\simeq \PP(L_+\oplus L_-)$. In addition, this geometrically
ruled surface is endowed with a parabolic structure deduced from the
parabolic structure on $\cM_1$ as explained at \S\ref{sec:tc}. In more
concrete terms, we pass from a parabolic structure on $\cM_1$ to a
parabolic structure on $\cM_0$ as follows:
let  $x^1_j$ be a
parabolic point in $\cM_1\simeq \cX$ such that
$\pi_\Sigma(x^1_j)=y_j$. Then, there is a parabolic point $x_j^0\in\cM_0$ in
the fiber of $y_j\in \Sigma$, such that  $x^0_j\in
S_+$ if $x^1_j\in S$, and, $x^0_j\in S_-$ otherwise. Eventually, the parabolic
weight attached to $x_j^0$ is given by the weight of $x^1_j$.
The central fiber $\cMhat_0$ of the test configuration given by Proposition
\ref{prop:tc} is the iterated blowup of the parabolic ruled surface
$\cM_0\to \Sigma$. Similarly  to $\cXbar\to \Sbar$,
we obtain a complex geometrically ruled orbifold surface
$\cMbar_0\to\Sbar$ by contracting the $E^\pm_j$-curves
in $\cMhat_0$.

By construction $\cM_0$ is endowed with a $\CC^*$-action coming from
the $\CC^*$-action on the complex manifold $\cM$. In fact, this action
is determined by the following properties:
\begin{itemize}
\item the action is free on a dense open subset of $\cM_0$,
\item it preserves the fibers of the ruling $\cM_0\to\Sigma$,
\item  the 
sections $S_\pm$ are the fixed points of the action,
\item the points of $S_-$ are repulsive, and the points of $S_+$ are attractive.
\end{itemize}
  
As all the parabolic points of $\cM_0\to\Sigma$ belong to $S_-\cup
S_+$, 
the $\CC^*$-action lifts to 
 $\cMhat_0\to \cM_0$.  Let $\hat S_-$ and $\hat
S_+$ be the proper
transforms of $S_-$ and $S_+$ in $\cMhat_0$. 
Notice that the $\CC^*$-action also descends via the canonical
projection $\cM_0\to \cMbar_0$, 
since it must preserve holomorphic spheres of negative self-intersection.
\subsection{Cremona transformations}
\label{rk:parab}
As we noticed, any parabolic point $x\in
F=\pi_{\Sigma}^{-1}(y)\subset\cM_0$ belongs to $S_+$ or $S_-$.
Assume $x\in
  S_-$. Let $ \cM_0''\to\cM_0$ be the blowup at 
  $x$ and $\hat F\subset \cM_0''$ be the proper transform of $F$.
Since $\hat F$  has self-intersection $-1$ it can
  be contracted back to a point $x'$. The contraction is denoted $\cM_0''\to\cM_0'$. Such an operation (blowing up,
  then contracting) is called a \emph{Cremona transformation}. Notice that the
  proper transform
  $S_+'\subset \cM_0'$ of $S_+$ contains the point $x'$.

Furthermore, the ruled surface $\cM_0'\to \Sigma$ has a
  natural parabolic structure induced by the parabolic structure of
  $\cM_0\to\Sigma$ with the convention that $x$ has been replaced by
  $x'$ and the corresponding weight $\alpha$ is now replaced by
  $\alpha'=1-\alpha$. It is an easy exercise to show that the notions
  of parabolic stability are invariant under such Cremona
  transformation. In addition the iterated blowup encoded by either
  parabolic ruled surfaces are both $\cMhat_0$.

Therefore, we may assume that all the parabolic points of $\cM_0$
belong to $S_+$ 
after performing a finite number of Cremona transformations. The
condition of stability is unchanged provided the weights are modified
according to the above convention.

\subsection{Weights of the $\CC^*$-action along special fibers}
The fibers of
$\cMhat_0\to \Sigma$ are preserved by the $\CC^*$-action. 
Generic
fibers are identified to $\CP^1$ endowed with a $\CC^*$-action of
weight $1$, and the two
fixed points correspond to the intersections of the fiber with $\hat S_\pm$.
In contrast, blownup fibers have more complicated $\CC^*$-action.
We start from $\pi_\Sigma:\cM_0\to \Sigma$ and assume that there is only one
parabolic point $x$ for simplicity. By \S\ref{rk:parab} we may
also assume that $x\in S_+$. The fiber $F$ containing $x$ is
represented by the configuration of curves
\begin{equation}
\xymatrix{
{}\ar@{-}'[d]_{S_-}  
       '[dr]^{0}_{\underset F1} 
       '[r]_{S_+} & \\
*+[o][F]{} &  *{\bullet} 
}
\end{equation} 
Here, the black dot represents the point $x$ in the fiber
$F=\pi^{-1}_\Sigma(y)$ of self-intersection $0$.  The integer $1$ represents
the weight of the 
$\CC^*$-action induced on $F$.

Then we blowup the point $x$ and get a configuration
\begin{equation}
\label{eq:weightstart}
\xymatrix{
{}\ar@{-}[d]_{\hat S_-} & & {}\ar@{-}[d]^{\hat S_+} \\
*+[o][F]{}\ar@{-}[r]^{-1}_{\underset {\hat F}1} & *+[o][F-]{}
\ar@{-}[r]^{-1}_{\underset {\hat E}1} &  *+[o][F]{}
}
\end{equation}
In the above diagram, the integer $1$ represent the weights of the
induced $\CC^*$-actions on the proper transform $\hat F$ of $F$ and on
$\hat E$ the exceptional divisor of the blowup.
Using the same notation, we blowup the intersection of the $-1$ curve
and obtain
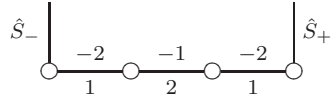
\begin{equation}
\xymatrix{
{}\ar@{-}[d]_{\hat S_-} & & & {}\ar@{-}[d]^{\hat S_+} \\
*+[o][F]{}\ar@{-}[r]^{-2}_1 &*+[o][F]{}\ar@{-}[r]^{-1}_2 & *+[o][F-]{}
\ar@{-}[r]^{-2}_1 &  *+[o][F]{}
}
\end{equation}

The we iterate our blowup procedure in order to get a diagram of the form
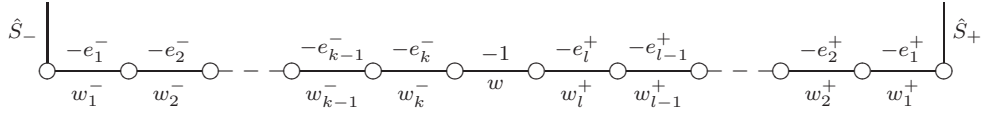
\begin{equation}
\label{eq:bupw}
\xymatrix{
\ar@{-}[d]_{\hat S_{-}}&&&&&&&&&&&\ar@{-}[d]^{\hat S_+} \\
 *+[o][F-]{}\ar@{-}[r]^{-e^-_1}_{w^-_1} & *+[o][F-]{}
\ar@{-}[r]^{ -e^-_{2}}_{w^-_2} &  *+[o][F-]{}
\ar@{--}[r] &  *+[o][F-]{}
\ar@{-}[r]^{-e^-_{k-1}}_{w_{k-1}^-} &  *+[o][F-]{}
\ar@{-}[r]^{-e^-_k}_{w^-_k} &  *+[o][F-]{}
\ar@{-}[r]^{-1}_{w} &  *+[o][F-]{}
\ar@{-}[r]^{-e^+_{l}}_{w^+_l} &  *+[o][F-]{}
\ar@{-}[r]^{-e ^+_{l-1}}_{w^+_{l-1}} &  *+[o][F-]{}
\ar@{--}[r] &  *+[o][F-]{}
\ar@{-}[r]^{-e^+_{2}}_{w^+_2} &  *+[o][F-]{}
\ar@{-}[r]^{ -e^+_{1}}_{w^+_1}&  *+[o][F-]{}
}
\end{equation}
The weight of the $\CC^*$-action induced on the $-1$-curve is computed
by induction, using the simple formula $w=w^-_k+w^+_l$.  We shall also
use the notation $E^\pm_j$ for the curve of self-intersection
$-e^\pm_j$ and $E_0$ for the $-1$-curve. 

Instead of starting with the configuration \eqref{eq:weightstart}, we
can formally replace the weights with the new configuration
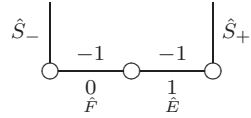
\begin{equation*}
\label{eq:weightstart2}
\xymatrix{
{}\ar@{-}[d]_{\hat S_-} & & {}\ar@{-}[d]^{\hat S_+} \\
*+[o][F]{}\ar@{-}[r]^{-1}_{\underset {\hat F}0} & *+[o][F-]{}
\ar@{-}[r]^{-1}_{\underset {\hat E}1} &  *+[o][F]{}
}
\end{equation*}
Using the same induction as for $w^\pm_j$, we construct a weight
system 
$$v_1^-,
\cdots, v_k^-,v,v^+_l,\cdots, v_1^+.$$

By definition of the weights and the ajunction, we have
$$
F=wE_0+ \sum_{n=1}^kw_n^-E_n^- + \sum_{n=1}^lw_n^+E_n^+
$$
and
$$
\hat E=vE_0+ \sum_{n=1}^kv_n^-E_n^- + \sum_{n=1}^lv_n^+E_n^+,
$$
where $\hat E$ and $F$  denote the pullback the homology classes
to $\cMhat_0$.

We gather the relevant results in the following lemma:
\begin{lemma}
\label{lemma:weights}
Using the convention $w=w^-_{k+1}=w^+_{l+1}$, we have
  $$
\sum_{n=1}^k \frac 1{w^-_nw^-_{n+1}} = \alpha, \quad \sum_{n=1}^l \frac 1{w^+_nw^+_{n+1}} = 1-\alpha.
$$
using the notation $\alpha=p/q$,  the weights introduced above satisfy 
$$w=q,\quad v=p,\quad w^\pm_1=1,\quad v^+_1=1 \quad \mbox{ and } v^-_1=0.
$$ 
\end{lemma}
\begin{proof}
The proof by induction is straightforward and left as an exercice for the interested reader.
\end{proof}

\subsection{Computation of the Futaki invariant}
 LeBrun \emph{et al} \cite{LebSin93,LebSim94} computed the Futaki
invariant of a ruled surface endowed with a semi-free $\CC^*$-action.
We are going to point out what should be modified for a general action. The reader is strongly advised
to refer to \cite[Section 3.3]{LebSim94} as we are following closely
their notations.

Let $\Xi$ be the $(1,0)$-holomorphic vector field  on $\cMhat_0$ defined
as the Euler vector field that generates the $\CC^*$-action. We define
$\xi= -2\Im \Xi$ as the (real) vector field that spans the underlying
circle action. 
Let $\omega$ be a circle-invariant Kähler form on $\cMhat_0$ with
Kähler class $\Omega=[\omega]$.
Since $\xi$ is a Killing field vanishing at some point, it is automatically
Hamiltonian (cf. \cite{LebSim94}). In other words, there exists a
smooth Hamiltonian function $t:\cMhat_0\to \RR$ such that
$dt=-\iota_{\xi}\omega$. Then $t$ admits a minimum along
$\hat S_-$ and a maximum along $\hat S_+$.
Up to adding a suitable constant, we may assume that
$t:\cMhat_0\to [-a,+a]$ is a surjective map for some $a>0$ and  $\hat
S_\pm=t^{-1}(\pm a)$.

Again, we are assuming that $\cM_0$ has only one parabolic point $x$
 to keep notations simple. Up to a Cremona tranformation, we may even
 assume that $x\in S_+$ (cf. \S\ref{rk:parab}).  The set of critical points of the function $t$ on $\cMhat_0$ consists of
the divisors $\hat S_\pm$ where $t$ is extremal,  and isolated  sadle
points. The latter are given by the intersections of the
$E^\pm_j$ and $E_0$-curves. These
points represented by the hollow dots in Diagram~\eqref{eq:bupw}. It
will be convenient to label them
$f_0,\cdots,f_k,f_{k+1},\cdots,f_{k+l+1}$ from the left to the
right. In the same spirit we shall use a notation $w_1=w_1^-, \cdots,
w_k=w_k^-,w_{k+1}=w,w_{k+1}=w^+_l,\cdots, w_{k+l+1}=w^+_1$.

 Let $\cY= \cMhat_0\setminus (S_+\cup S_-\cup \{f_j\})$, the regular locus of
$t$. By definition $\cY$ is a Seifert manifold. Any
point $z\in \cY$ has trivial stabilizer, unless $z$ belongs to $E_0$
or $E^\pm_j$ where the stabilizer is the cyclic group respectively of order
$w$ and $w^\pm_j$. Hence the quotient $N=\cY/S^1$ has an orbifold
structure
 and $\varpi:\cY\to N$ is an orbifold circle bundle. This is the main
 difference with the case of a semi-free action, where $\cY\to N$
 is a
 smooth circle bundle. 

The function $t$ is
invariant under the circle action, hence the fibers $\cY_c$ of
$t:\cY\to (-a,a)$ are endowed with a circle action and the map
descends to $t:N\to (-a,a)$. 
If $c$ is a regular value of $t$, 
$N_c= t^{-1}(c)\subset Y$ is a compact orbifold.
Moreover, $N_c$  has a natural Kähler structure since
it is a Kähler moment map reduction of $(\cMhat_0,\omega)$ by 
the Hamiltonian action of the circle. Furthermore $N_c$ is
isomorphic to $N_{d}$ if there are no critical value in the
interval $[c,d]$.

Let $t_j=t(f_j)$ be the $t$-coordinate of the fixed point $f_j$. 
Then the following facts hold, by definition:
\begin{enumerate}
\item If $t_{j}<c<t_{j+1}$ for some $0\leq j\leq k+l$, the Riemann surface $N_c$
is isomorphic to $\Sigma$, where the marked point $y$ of the parabolic
structure has been replaced by an orbifold point of order $w_{j+1}$. In
particular $N_c\simeq \Sigma$ is $j=0$ or $k+l$ and $N_c\simeq \Sbar$
if $j=k$. 
\item
Let $S_j\subset N$ be a small sphere (with orbifold singularities) centered at a point $\varpi(f_j)$, for
$1\leq j\leq k+l$. Then the orbicircle bundle $\cY|_{S_j}\to S_j$ has
orbifold degree 
\begin{equation}
 c_1(\cY)\cdot [S_j] = \frac 1{w_jw_{j+1}}.
\end{equation}
 This readily seen as $\cY\to S_j$ admits a $w_jw_{j+1}$-fold ramified
 cover by the Hopf fibration $S^3\to S^2$. This is also an
essential difference with the case of a semi-free circle action~\cite[top of the page 315]{LebSim94}.
\end{enumerate}

It is also convenient to use a rescaled Kähler class in comparison with
\S\ref{sec:bkahl}. Here we shall assume that the Kähler class $\Omega$
satisfies the identity $\Omega\cdot F =1$.
Adapting carefully the computation of \cite[p. 318-319]{LebSim94} to this orbifold
context, and relying on facts (1) and (2) above, we get the identity
\begin{equation}
  \label{eq:intt}
  \int_{\cMhat_0} t \;d\mu = \frac 1{96\pi}\left ( \hat S_-^2 - \hat S_+^2 + 6
    \Omega\cdot \left ( \hat S_+ - \hat S_-\right ) - 64\pi^3 \sum_{j=1}^{k+l}\frac
    {t_j^3}{w_jw_{j+1}}\right )
\end{equation}
where $d\mu$ is the volume form of the Kähler metric $\omega$.
We also have the modified formula
\begin{equation}
  \label{eq:intst}
\int st \; d\mu =\Omega\cdot\left (\hat S_+-\hat S_-\right ) +
4\pi^2\sum_{j=1}^{k+l}(w_j^{-1}-1)(t_{j+1}^2-t_j^2) 
\end{equation}
where $s$ is the scalar curvature of the metric.  By definition of the
Futaki invariant $\Fut(\Xi,\Omega)=\int (\bar s^\Omega-s)t\;d\mu$ and we end up with the
formula 
\begin{align*}
\Fut(\Xi,\Omega) =&  \Omega\cdot (\hat S_--\hat S_+) -
4\pi^2\sum_{j=1}^{k+l}(w_j^{-1}-1)(t_{j+1}^2-t_j^2) \\
 + & \frac {\bar s^\Omega}{96\pi}\left ( \hat S_-^2 - \hat S_+^2 + 6
    \Omega \cdot ( \hat S_+ - \hat S_-) - 64\pi^3 \sum_{j=1}^{k+l}\frac
    {t_j^3}{w_jw_{j+1}}\right )  
\end{align*}
where
$$
\bar s^\Omega = \int s \;d\mu = 8\pi \frac{c_1(\cMhat_0)\cdot \Omega}{\Omega^2}.
$$

\subsection{A computation in the degenerate case}
Notice that if we let $\Omega$  degenerates toward (the pullback of)
an orbifold Kähler 
class $\Omega^{orb}$ on $\cMbar_0$,
we have $t_1=\cdots =t_k=-a$ and $t_{k+1}=\cdots=t_{k+l}=a$. Therefore,
using Lemma \ref{lemma:weights} and the fact that $4\pi a=\Omega\cdot
F=1$ (cf. \cite[bottom of p. 315]{LebSim94}), we obtain
$$
\lim_{\Omega\to\Omega^{orb}}  \int t\; d\mu = \frac 1{96\pi}\left ( \hat S_-^2 - \hat S_+^2 + 6
    \Omega\cdot ( \hat S_+ - \hat S_-) + \alpha - (1-\alpha) \right )
$$
On the other hand $[\hat S_+]^2=[S_+]^2-1$ and $[\hat S_-]^2=[S_-]^2$ since the first blowup occurred
at $x\in S_+$. It follows that 
$$[\hat S_-]^2 +\alpha =\parmu (S_-), \quad \mbox{ and }\quad  \hat S_+^2+1-\alpha =\parmu(S_+).$$
Finally $\Omega\cdot (\hat S_+-\hat S_-)= \Omega^{orb}\cdot (\bar S_+-\bar
S_-)$ for a orbifold Kähler class, where $\bar S_\pm=\pi_{\cMbar_0}(\hat
S_\pm)$. The holomorphic sections $\bar S_\pm$ of $\cMbar_0\to\Sbar$
corresponds to orbifold line 
bundle  $\bar L_\pm\to \Sbar$. In this context, it is well known that
(cf. for instance \cite{FurSte92})
$$\orbic (\cO_{\cMbar_0}(1))\cdot 
\bar S_\pm = \orbideg \bar L_\pm = \pardeg L_\pm,$$
where $\orbideg$ is the natural notion of degree for an orbifold line bundle.
Hence   $\Omega\cdot (\hat S_+-\hat S_-)= \pardeg L_+ -
\pardeg L_-= (\pardeg L_+ +
\pardeg L_-)- 2\pardeg L_-=\parmu (S_-)=-\parmu(S_+)$

In conclusion, we have the following result:
\begin{lemma}
Let $\Omega\in\Kc(\cMhat_0)$ and $\Omega^{orb}\in \Kc(\cMbar_0)$
considered as a class on $\cMhat_0$ as well. Then
$$\lim_{\Omega\to\Omega^{orb}}\int_{\cMhat_0} t \;d\mu = - \frac {\parmu (S_+)}{12\pi},
\quad \lim_{\Omega\to\Omega^{orb}}\int_{\cMhat_0} st \; d\mu = -\parmu(S_+).
$$
\end{lemma}

\subsection{Sign of the Futaki invariant}
\label{sec:sfut}
We deduce the following lemma, which will be crucial for the proof of
Proposition~\ref{prop:futsign}.
\begin{lemma}
\label{lemma:futs1}
Suppose that $\parmu(S_+)\neq 0$.
  There exists a sufficiently small open cone
  $U\subset H^2(\cMhat_0,\RR)$ containing the ray $\Kr$, such that
  for every Kähler class $\Omega\in  U\cap \Kc(\cMhat_0)$, the Futaki invariant
  $\Fut(\Xi,\Omega)$ does not vanish and has the same sign as $\parmu(S_+)$.
\end{lemma}
\begin{proof}
  We start with an orbifold Kähler class $\Omega^{orb}_C=\orbic (\cO_{\cMbar_0}(1))+CF$ on
  $\cMbar_0$, where $C>0$ is chosen very large.
We use a generalization of the classical result for
smooth geometrically ruled surfaces:
$$\orbic (K_{\cMbar_0})= -2\orbic (\cO_{\cMbar_0}(1)) +
(\pardeg (E) -
\chiorb(\Sbar)) F,$$
where $K_{\cMbar_0}$ is the (orbifold) canonical line bundle of the
orbifold $\cMbar_0$  and $\chiorb(\Sbar)$ is the orbifold Euler
characteristic given by 
$$
\chi^{orb}(\Sbar)=\chi(\Sbar) + \sum_{j=1}^m (\frac 1{q_j}-1).
$$

It follows that  $\bar s^{\Omega^{orb}_C} = 8\pi \frac {\pardeg(E)+\chiorb(\Sbar)+2C}{
  \pardeg(E)+2C}$. In particular, we see that
$$
\lim_{C\to +\infty} \bar s^{\Omega^{orb}_C} = 8\pi.
$$

Using the fact that
$$
\lim_{\Omega\to \Omega_c^{orb}}s^{\Omega}_0 =
s^{\Omega^{orb}_C}_0\quad\mbox {and}\quad 
\lim_{\Omega\to \Omega_C^{orb}} \Omega \cdot \hat S_\pm = 
\Omega^{orb}_C\cdot \bar S_\pm.
$$
and that the corresponding values of $t_j$ converge in the
following way
$$  t_j\to -a \mbox{ for $j\leq k$, and } \quad
t_j\to a  \mbox{ for $j\geq k+1$}
$$
as $\Omega\to\Omega^{orb}$,
we see deduce that

\begin{equation}
\label{eq:limom}
\lim_{\Omega \to \Omega^{orb}_C}\Fut( \Xi,\Omega)= \left (1 - \frac
{\bar s^{\Omega^{orb}_C}}{12\pi} \right ) \parmu(S_+).   
\end{equation}
For $C$ sufficiently large, the coefficient in front of
$\parmu(S_+)$ is positive since
$$
\lim_{C\to +\infty}\left ( 1 - \frac
{\bar s^{\Omega^{orb}_C}}{12\pi}\right ) = \frac 13
$$
 and the lemma follows.
\end{proof}

Eventually, we deal with the case where the section has vanishing
slope.
The case of a trivial parabolic structure is slightly different and
must be treated separately.

\begin{lemma}
\label{lemma:futs3}
  With the above notations, suppose that $\parmu (S_+)=0$ and that the
  parabolic structure of $\cM_0\to \Sigma$ is empty. 
Then 
the Futaki invariant  $\Fut(\Xi,\cdot)$  vanishes for every Kähler
class on $\cMhat_0=\cM_0$.
\end{lemma}
\begin{proof}
Here $\cM_0\to\Sigma$ has two non-intersecting holomorphic $S_\pm$ with
vanishing slope. Hence $\cM_0\to\Sigma$ is polystable in the usual
  Mumford sense and $\cMhat_0=\cM_0$ as the parabolic structure is
  empty. It follows that 
  $\cM_0\to\Sigma$ is a flat projective bundle by the
  Narasimhan-Seshadri theorem \cite{NarSes64}. This implies that every
  Kähler class of $\cM_0$ can be represented by a CSCK metric, obtained
  as a local product of metrics of constant curvature. Therefore, the
  Futaki invariant vanishes identically.
\end{proof}

\begin{lemma}
\label{lemma:futs2}
  With the above notations, suppose that $\parmu (S_+)=0$ and that the
  parabolic structure is not empty. Then for
  every open cone $U$ in $H^2(\cMhat_0,\RR)$ such that $\Kr\subset
  U$, there are rational Kähler classes $\Omega\in U\cap \Kc(\cMhat_0)$
  such that $\Fut(\Xi,\Omega)>0$ and such that  $\Fut(\Xi,\Omega)<0$.

In addition the equation $\Fut(\Xi,\cdot)=0$ cuts  $U\cap
\Kc(\cMhat_0)$ along a nonempty regular hypersurface.
\end{lemma}
\begin{proof}
By \eqref{eq:limom}, we have $\lim_{\Omega\to\Omega^{orb}_C}\Fut(\Xi,
\Omega)=0$. 
This corresponds to
the limiting value of the Futaki invariant when taking  the parameters $a=t^+_j=-t^-_j$ for
$j\geq 1$.
The idea to prove the  Lemma is to compute the partial
derivatives of the Futaki invariant at $\Omega^{orb}_C$. In fact it
will suffice to consider variations of the Kähler class corresponding
to the parameters 
$t_j^-= \frac {\tau^-}{4\pi}-a$ and $t_j^+=a-\frac{\tau^+}{4\pi}$  for
$j\geq 1$ for $\tau^\pm>0$ sufficiently small.
By definition  the corresponding (orbifold) class $\Omega$ satisfies
$\Omega\cdot E_1^\pm =\tau^\pm$ and  $\Omega\cdot E^\pm_j = 0$ for $j>1$.
Using Lemma \ref{lemma:weights} and the fact that $\Omega\cdot F =1$, we find 
$\Omega\cdot E_0 = \frac {1-(\tau^++\tau^-)}q$.

By the adjunction formula $\hat S_+ = S_+-\hat E$. Furthermore, in
terms of Poincaré dual, $S_\pm
= c_1(\cO_{\cM_0}(1)) - \deg (L_\pm) F$. Thus $\hat S_+-\hat S_- = (\deg (L_-)-\deg (L_+))F
-\hat E$ and it follows that $\Omega\cdot (\hat S_+-\hat S_-) = \deg
L_- - \deg L_+ - \tau ^+ - \alpha(1-(\tau^++\tau^-)) = \pardeg
L_--\pardeg L_+ +\alpha\tau^-+(\alpha-1)\tau^+$. By assumption $L_+$
and $L_-$ have the same parabolic degree, therefore 
$$
\Omega\cdot (\hat S_+-\hat S_-) = \alpha\tau^-+(\alpha-1)\tau^+.
$$
Finally, the formula for the Futaki invariant for a variation
$\tau^\pm$ can be written
\begin{multline*}
\Fut(\tau^\pm) =  (1-\frac 6{96\pi}\bar s^\Omega)((1-\alpha)\tau^+-\alpha\tau^-)\\
-64\pi^3\frac {\bar s^\Omega}{96\pi}\left (\sum_{j=1}^k\frac
  {(t^-_j)^3+a^3}{w_j^-w_{j+1}^-} + \sum_{j=1}^l\frac
  {(t^+_j)^3-a^3}{w_j^-w_{j+1}^-}
\right )
\end{multline*}

If $\Omega$ is a sufficiently small perturbation of $\Omega^{orb}_C$
and $C>0$ is large enough, then $\bar s^\Omega$ is very close to $8\pi$. It
follows that the Futaki invariant is of the form
$\Fut(\tau^\pm)=C_1f_1 +C_2f_1$ where $C_1,C_2>0$,
$f_1=(1-\alpha)\tau^+-\alpha \tau^-$ and
$$
f_2=-\sum_{j=1}^k\frac
  {(t^-_j)^3+a^3}{w_j^-w_{j+1}^-} - \sum_{j=1}^l\frac
  {(t^+_j)^3-a^3}{w_j^-w_{j+1}^-}
$$ 
The differential are easily computed at $\tau^\pm=0$ and they are
positive multiples of
$$
 (1-\alpha)d\tau^+ -\alpha d\tau^-.
$$

It follows that the variation of $\frac {\del f_j} {\del\tau^-} <0$. In
particular, $f_j$ are negative for certain arbitrarily small values of
$\tau^\pm>  0$.
It follows that $\Fut(\tau^\pm)$ must be negative as well.
Similarly  $\frac {\del f_j} {\del\tau^+} >0$ and the Futaki invariant also
take positive values for certain values of $\tau^\pm> 0$ arbitrarily
small.

By density, we can always assume that the parameters $C$ and
$\tau^\pm$ are chosen suitably so that $\Omega$ is rational. We can
also perturb the cohomology class $\Omega$ by higher order terms so
that it is a Kähler class on $\cMhat_0$ (and not just an orbifold
Kähler class). The
first part of the lemma follows.

Notice that our computation shows  that the
Futaki invariant is a submersion vanishing at $\Omega^{orb}_C$. So the
surface given by the vanishing of the Futaki invariant is regular near
$\Omega_C^{orb}$. The
fact that the partial derivatives $\frac {\del \cF} {\del\tau^\pm}$ have
opposite signs insures that the surface $\Fut(\Xi,\cdot)=0$ cuts the
quadrant $\tau^\pm>0$ along a nonempty set, hence cuts the Kähler cone and the lemma follows. 
\end{proof}

\bibliographystyle{abbrv}
\bibliography{kpstab,rollin}

\end{document}